\documentclass[12pt, one side,emlines]{amsart}
\usepackage[a4paper,margin=1in]{geometry}
\usepackage{fancyhdr}
\usepackage{enumerate}
\usepackage{amssymb,latexsym,xy,eucal,mathrsfs}
\textwidth=18cm \textheight=27cm \theoremstyle{plain}
\usepackage{tikz}
\usetikzlibrary{positioning}
\theoremstyle{definition}

\newtheorem{theorem}{Theorem}[section]
\newtheorem{thm}[theorem]{Theorem}
\newtheorem{lem}[theorem]{Lemma}

\newtheorem{defn}{Definition}[section]

\begin{document}
	

	%
	\title{Prohibited Minors for Graphic Matroids that gives a Binary Gammoid after Splitting}\maketitle
	
	\markboth{ Shital Dilip Solanki and S. B. Dhotre}{Prohibited Minors for Graphic Matroids that gives a Binary Gammoid after Splitting}\begin{center}\begin{large} Shital Dilip Solanki$^1$ and S. B. Dhotre$^2$\end{large}\\\begin{small}\vskip.1in\emph{
			1. Ajeenkya DY Patil University, Pune-411047, Maharashtra,
			India\\ 
			2. Department of Mathematics,
			Savitribai Phule Pune University,\\ Pune - 411007, Maharashtra,
			India}\\
		E-mail: \texttt{1. shital.solanki@adypu.edu.in, 2. dsantosh2@yahoo.co.in. }\end{small}\end{center}\vskip.2in
\begin{abstract} 
	Splitting operation in Matroid Theory does not preserve graphicness, connectedness, cographicness, etc. Also, the splitting of binary gammoid does not necessarily be binary gammoid after splitting. We have characterized a class of graphic matroids that gives binary gammoids after splitting. We have obtained prohibited minors for graphic and cographic matroid which gives binary gammoid after splitting using two and three elements. 

	\end{abstract}\vskip.2in
\noindent\begin{Small}\textbf{Mathematics Subject Classification (2010)}:
05B35,05C50, 05C83.      \end{Small}\vskip.2in
\vskip.25in

\baselineskip 19truept 
\section{Introduction}

\noindent Refer to Oxley \cite{ox}, for unexplained concepts in the matroid theory.\\
\noindent
In graphs, the splitting operation is well known and is introduced by Fleischner \cite{fl}. The splitting of a graph using two arcs is shown in Figure \ref{Fig_sp_graphs}. Let $J$ be a graph with two arcs $p=(v_1,v)$ and $q=(v_2,v)$ incident at node $v$ as given in Figure \ref{Fig_sp_graphs}. Then the graph $J_{p,q}$ obtained by removing edges $p$, $q$ and adding a new vertex $v'$ and edges $p=(v_1,v')$ and $q=(v_2,v')$. The graph $J_{p,q}$ is splitting of $J$ using $\{p,q\}$.

\begin{figure}[h!]
	\centering
\unitlength 1mm 
\linethickness{0.4pt}
\ifx\plotpoint\undefined\newsavebox{\plotpoint}\fi 
\begin{picture}(78.975,26.906)(0,0)
	\put(22.38,10.007){\circle*{1.35}}
	\put(4.83,10.035){\circle*{1.35}}
	\put(4.693,10.051){\line(1,0){17.581}}
	\put(4.976,24.182){\circle*{1.35}}
	\put(22.361,24.407){\circle*{1.35}}
	\put(4.873,24.274){\line(1,0){17.392}}
	\put(22.265,24.274){\line(0,-1){14.047}}
	\put(4.74,9.958){\line(0,1){14.717}}
	\multiput(22.326,24.632)(-.0409585253,-.0336981567){434}{\line(-1,0){.0409585253}}
	\multiput(4.835,24.386)(.0405790698,-.0337325581){430}{\line(1,0){.0405790698}}
	\put(34.362,10.135){\circle*{1.35}}
	\put(34.343,24.535){\circle*{1.35}}
	\put(34.247,24.402){\line(0,-1){14.047}}
	\put(22.494,24.386){\line(1,0){12.403}}
	\put(22.074,9.881){\line(1,0){12.614}}
	\put(16.187,26.278){\makebox(0,0)[cc]{$p$}}
	\put(18.71,18.71){\makebox(0,0)[cc]{$q$}}
	\put(23.125,26.699){\makebox(0,0)[cc]{$v$}}
	\put(3.153,26.699){\makebox(0,0)[cc]{$v_1$}}
	\put(2.523,8.829){\makebox(0,0)[cc]{$v_2$}}
	\put(18.92,5){\makebox(0,0)[cc]{$J$}}
	\put(66.317,9.797){\circle*{1.35}}
	\put(48.767,9.825){\circle*{1.35}}
	\put(48.63,9.841){\line(1,0){17.581}}
	\put(48.913,23.972){\circle*{1.35}}
	\put(66.298,24.197){\circle*{1.35}}
	\put(66.202,24.064){\line(0,-1){14.047}}
	\put(48.677,9.748){\line(0,1){14.717}}
	\multiput(48.772,24.176)(.0405790698,-.0337325581){430}{\line(1,0){.0405790698}}
	\put(78.3,9.925){\circle*{1.35}}
	\put(78.281,24.325){\circle*{1.35}}
	\put(78.185,24.192){\line(0,-1){14.047}}
	\put(66.431,24.176){\line(1,0){12.403}}
	\put(66.011,9.67){\line(1,0){12.614}}
	\put(47.091,26.488){\makebox(0,0)[cc]{$v_1$}}
	\put(46.46,8.619){\makebox(0,0)[cc]{$v_2$}}
	\put(60.351,24.256){\circle*{1.35}}
	\put(48.683,24.305){\line(1,0){11.67}}
	\multiput(60.353,24.23)(-.0336504298,-.0413151862){349}{\line(0,-1){.0413151862}}
	\put(60.65,26.906){\makebox(0,0)[cc]{$v'$}}
	\put(56.19,26.46){\makebox(0,0)[cc]{$p$}}
	\put(66.447,26.906){\makebox(0,0)[cc]{$v$}}
	\put(63.177,5){\makebox(0,0)[cc]{$J_{p,q}$}}
	\put(59.164,19){\makebox(0,0)[cc]{$q$}}
\end{picture}
	
\caption{Splitting operation in graphs}
\label{Fig_sp_graphs}	
\end{figure}
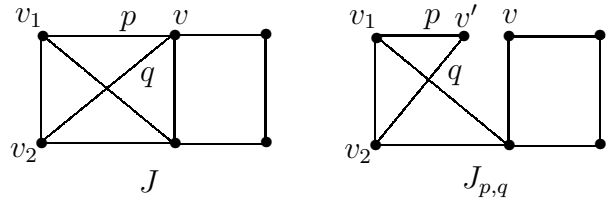
\noindent In \cite{ttr}, Raghunathan et al. introduced a splitting operation using two elements for binary matroids which was then generalized by Shikare et al. \cite{mms1} for $n$-elements. A definition for generalized splitting is given below.

\begin{defn}\cite{mms1}\label{def_Splitting}
	$B$ is a binary matroid such that it is represented by a matrix $A$. Obtain a new matrix $A_X$, for $X \subseteq E(B)$, by placing a new row after the last row of $A$ with entries $1$ in the columns representing elements of $X$ and remaining entries $0$. Then the matroid $B_X=M(A_X)$ is the splitting matroid and the splitting using a set is the transition from $B$ to $B_X$.
\end{defn}  
\noindent Later, the element splitting was introduced by Azadi \cite{azt} for binary matroids.
\begin{defn}\cite{azt} \label{def_element_sp}
	$B$ is a binary matroid such that it is represented by a matrix $A$. Obtain a matrix $A_X'$, for $X \subseteq E(B)$, by placing a new row after the last row of $A$ with entries $1$ in the columns representing elements of $X$ and remaining entries $0$. Also, adding one column labeled $q$ with entry $1$ in the new row and remaining entries $0$. Then the matroid $B_X'=M(A_X')$ is the element splitting matroid, and the element splitting operation is the transition from $B$ to $B_X'$.
\end{defn}
\noindent From the above two definitions, is it clear that $B_X'\backslash q=B_X$ and $B_X'/q=B$.
\noindent In addition to this, the es-splitting for binary matroids was introduced by Azanchilar \cite{azn}.
\begin{defn}\label{def_es}
	A binary matroid $B$ is such that $B$ is represented by a matrix $A$. For $X \subseteq E(B)$ with $e \in X$, obtain a matrix $N$ by adding a column to $A$ labeled $\gamma$, similar to a column labeled by $e$. Let $D=M(N)$, then the es-splitting of $B$ is the element splitting matroid $D_X'$ of $D$ which is denoted by $B_X^e$. The transition from $B$ to $B_X^e$ is the es-splitting operation.  
\end{defn}

\noindent It was observed that the splitting, element splitting, and es-splitting of binary matroid do not preserve contentedness, graphicness, cographicness,  etc. Borse \cite{ymb1} obtained prohibited minors for a graphic matroid $M$ such that splitting of $M$ is cographic using $2$ elements. Also, Borse \cite{ymb_Gammoid} obtained forbidden minor for binary gammoids that gives a binary gammoid after splitting using $2$ elements. \\
Thus, we identify the prohibited minors for graphic matroids that give binary gammoids after splitting, element splitting, and es-splitting. We also characterized cographic matroids that give binary gammoids after splitting.

\noindent $\mathcal{GG}_k$ denotes the collection of graphic matroids whose splitting using $k$ elements is not a binary gammoid and $\tilde{M}$ denote a single element binary extension of $M$. In \cite{gm}, Mundhe et al. introduced a method of finding prohibited minors for graphic matroids which gives a graphic matroid after splitting. We use a similar technique to characterize graphic and cographic matroids whose splitting is gammoid.

\begin{thm}\label{main_cor}
	A splitting matroid $B_X$ is a binary gammoid for a graphic matroid $B$ and any $X \subseteq E(B)$ if and only if $B$ does not contain a $M(Q_i)$ minor, for $i=1,2,3,4$ and $|X| \geq 2$.
\end{thm}

\noindent We have obtained prohibited minors for graphic matroids which give a binary gammoid after splitting with respect to two elements as well as three elements. The theorems are as stated below.
\begin{thm}\label{thm_GGm2}
	A graphic matroid $B \in \mathcal{GG}_2$ if and only if $M(G_1)$ or $M(G_2)$ is a minor of $B$, Figure \ref{fig_GGm_2} shows the graphs $G_1$ and $G_2$.
\end{thm}
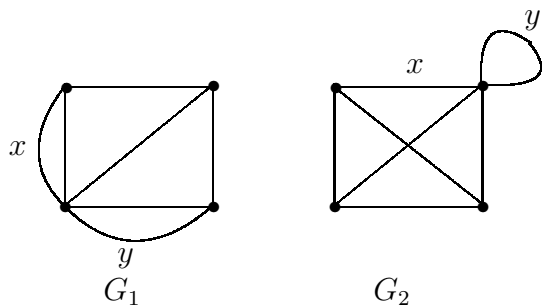
\begin{figure}[h!]
\centering
\unitlength 1mm 
\linethickness{0.4pt}
\ifx\plotpoint\undefined\newsavebox{\plotpoint}\fi 
\begin{picture}(75.375,38)(0,0)
	\put(28.81,11.25){\circle*{1.5}}
	\put(9.31,11.281){\circle*{1.5}}
	\put(9.158,11.299){\line(1,0){19.534}}
	\put(9.473,27){\circle*{1.5}}
	\put(28.79,27.25){\circle*{1.5}}
	\put(9.358,27.102){\line(1,0){19.325}}
	\put(28.683,27.102){\line(0,-1){15.608}}
	\put(9.21,11.196){\line(0,1){16.352}}
	\multiput(28.75,27.5)(-.0409751037,-.0337136929){482}{\line(-1,0){.0409751037}}
	\qbezier(9,11.5)(2.5,18.25)(9,27)
	\qbezier(9,11.5)(18.125,2)(28.75,11.5)
	\put(3,19){\makebox(0,0)[cc]{$x$}}
	\put(17.25,4.5){\makebox(0,0)[cc]{$y$}}
	\put(16.75,0){\makebox(0,0)[cc]{$G_1$}}
	\put(64.25,11.25){\circle*{1.5}}
	\put(44.75,11.281){\circle*{1.5}}
	\put(44.598,11.299){\line(1,0){19.534}}
	\put(44.913,27){\circle*{1.5}}
	\put(64.23,27.25){\circle*{1.5}}
	\put(44.798,27.102){\line(1,0){19.325}}
	\put(64.123,27.102){\line(0,-1){15.608}}
	\put(44.65,11.196){\line(0,1){16.352}}
	\multiput(64.19,27.5)(-.0409751037,-.0337136929){482}{\line(-1,0){.0409751037}}
	\multiput(44.75,27)(.0423913043,-.0336956522){460}{\line(1,0){.0423913043}}
	\qbezier(64,27)(63.25,38)(70.5,33)
	\qbezier(64,27.5)(75.375,26.75)(70.25,33)
	\put(55.25,30){\makebox(0,0)[cc]{$x$}}
	\put(70.75,36){\makebox(0,0)[cc]{$y$}}
	\put(52.25,0){\makebox(0,0)[cc]{$G_2$}}
\end{picture}
\caption{Minors of the class $\mathcal{GG}_2$}
\label{fig_GGm_2}
\end{figure}

\begin{thm}\label{thm_GGm3}
	A graphic matroid $B\in \mathcal{GG}_3$ if and only if $M(G_i)$ is a minor of $B$, Figure \ref{fig_GGm_3} shows the graph $G_i$, for $i=3,4,5,6$.
\end{thm}
\begin{figure}[h!]
	\centering
	\unitlength 1mm 
	\linethickness{0.4pt}
	\ifx\plotpoint\undefined\newsavebox{\plotpoint}\fi 
	\begin{picture}(105.583,32.375)(0,0)
		\put(49.129,10.053){\circle*{1.5}}
		\put(39.129,25.053){\circle*{1.5}}
		\put(29.629,10.084){\circle*{1.5}}
		\put(29.476,10.102){\line(1,0){19.534}}
		\multiput(38.913,25.203)(.0336763754,-.0485889968){309}{\line(0,-1){.0485889968}}
		\multiput(39.062,24.906)(-.0336655052,-.0512787456){287}{\line(0,-1){.0512787456}}
		\qbezier(39.211,25.055)(50.582,22.75)(49.17,10.041)
		\qbezier(29.548,10.189)(27.319,22.379)(38.765,25.055)
		\qbezier(38.913,25.352)(33.116,29.663)(39.211,31.001)
		\qbezier(39.211,31.001)(45.231,30.034)(39.062,25.203)
		\put(76.48,10.345){\circle*{1.5}}
		\put(66.48,25.345){\circle*{1.5}}
		\put(56.98,10.375){\circle*{1.5}}
		\put(56.828,10.394){\line(1,0){19.534}}
		\multiput(66.265,25.495)(.0336731392,-.0485889968){309}{\line(0,-1){.0485889968}}
		\multiput(66.413,25.198)(-.0336655052,-.0512787456){287}{\line(0,-1){.0512787456}}
		\qbezier(66.562,25.346)(77.934,23.042)(76.522,10.332)
		\qbezier(56.9,10.481)(54.67,22.67)(66.116,25.346)
		\qbezier(76.522,10.338)(83.509,11.156)(79.792,6.325)
		\qbezier(79.792,6.325)(74.59,3.426)(76.522,10.338)
		\put(104.129,10.345){\circle*{1.5}}
		\put(94.129,25.345){\circle*{1.5}}
		\put(84.629,10.375){\circle*{1.5}}
		\put(84.477,10.394){\line(1,0){19.534}}
		\multiput(93.914,25.495)(.0336731392,-.0485889968){309}{\line(0,-1){.0485889968}}
		\multiput(94.062,25.198)(-.0336655052,-.0512787456){287}{\line(0,-1){.0512787456}}
		\qbezier(94.211,25.346)(105.583,23.042)(104.171,10.332)
		\qbezier(84.549,10.481)(82.319,22.67)(93.765,25.346)
		\qbezier(84.549,10.338)(96.144,1.865)(104.171,10.635)
		\put(22.38,10.007){\circle*{1.35}}
		\put(4.83,10.035){\circle*{1.35}}
		\put(4.693,10.051){\line(1,0){17.581}}
		\put(4.976,24.182){\circle*{1.35}}
		\put(22.361,24.407){\circle*{1.35}}
		\put(4.873,24.274){\line(1,0){17.392}}
		\put(22.265,24.274){\line(0,-1){14.047}}
		\put(4.74,9.958){\line(0,1){14.717}}
		\multiput(22.326,24.632)(-.0409585253,-.0336981567){434}{\line(-1,0){.0409585253}}
		\multiput(4.835,24.386)(.0405790698,-.0337325581){430}{\line(1,0){.0405790698}}
		\put(12.824,2.312){\makebox(0,0)[cc]{$G_3$}}
		\put(39.102,2.312){\makebox(0,0)[cc]{$G_4$}}
		\put(66.221,2.312){\makebox(0,0)[cc]{$G_5$}}
		\put(94.391,2.312){\makebox(0,0)[cc]{$G_6$}}
		\put(16.818,26.068){\makebox(0,0)[cc]{$x$}}
		\put(15.557,20.812){\makebox(0,0)[cc]{$y$}}
		\put(20.182,17.238){\makebox(0,0)[cc]{$z$}}
		\put(31.324,23.335){\makebox(0,0)[cc]{$x$}}
		\put(39.102,32.375){\makebox(0,0)[cc]{$y$}}
		\put(49.193,21.653){\makebox(0,0)[cc]{$z$}}
		\put(59.914,25.017){\makebox(0,0)[cc]{$x$}}
		\put(75.261,23.545){\makebox(0,0)[cc]{$y$}}
		\put(80.937,4.625){\makebox(0,0)[cc]{$z$}}
		\put(84.3,21.443){\makebox(0,0)[cc]{$x$}}
		\put(104.902,21.443){\makebox(0,0)[cc]{$y$}}
		\put(94.601,7.568){\makebox(0,0)[cc]{$z$}}
	\end{picture}
	
	\caption{Minors of the class $\mathcal{GG}_3$}
	\label{fig_GGm_3}
\end{figure}
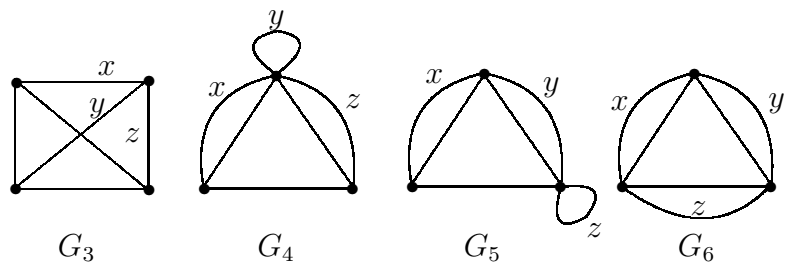
\noindent In section 4, we have characterized graphic matroids that result in a binary gammoid after element splitting and es-splitting. In the last section, we characterize cographic matroids whose splitting is a binary gammoid. 

\section{Preliminary Results}	
\noindent In the following theorem, Oxley \cite{ox} has described a binary gammoid.
\begin{thm}\label{gammoid}\cite{ox}
	For a matroid $B$ the below are equivalent.\\
	(i) $B$ is a graphic gammoid. \\
	(ii) $B$ is a regular gammoid. \\
	(iii) $B$ is a binary gammoid.\\
	(iv) $B$ has no minor isomorphic to $U_{2,4}$ or $M(K_4)$.
\end{thm}
\begin{thm}\label{graphic}\cite{ox}
	Let $B$ be a binary matroid. Then $B$ is graphic if and only if it does not contain minor $F \in \{F_7, F_7^*, M^*(K_5), M^*(K_{3,3})\}$.
\end{thm}

\noindent We proved the following results in this section which are useful to prove the main theorems.

\begin{lem}\label{GG_main_lem}
	Let $B$ be a graphic matroid such that $B_X$ has a $M(K_4)$ minor, for $X \subseteq E(B)$ with $|X|\geq 2$. Then $B$ has a minor $S$ containing $X$ which satisfies one of the below conditions. 
	\begin{enumerate}[(i).]
		\item $S_X \cong M(K_4)$;
		\item $S_X/X' \cong M(K_4)$, where $X'\subseteq X$;
		\item $S$ contains a minor $M(K_4)$.
		\item $S \cong \tilde{N}$ for some $N \in \mathcal{GG}_{k-1}$, for $k \geq 2$.
	\end{enumerate}  
\end{lem}
\begin{proof}
	Suppose $B$ be a graphic matroid and $B_X$ contain $M(K_4)$. Thus for some subsets $X_1$ and $X_2$ of $E(B)$, $B_X \backslash X_1 /X_2 \cong M(K_4)$.  Let $X_i'=X \cap X_i$ and let $X_i''=X_i-X_i'$, for $i=1,2$. Then $B_X\backslash X_1''/X_2'' \cong (B\backslash X_1''/X_2'')_X$ as each $X_i''$ is disjoint from $X$. Let $S=B\backslash X_1''/X_2''$, here $S$ is a minor of $B$. Consider $S_X\backslash X_1'/X_2' \cong (B\backslash X_1''/X_2'')_X \backslash X_1'/X_2'\cong B_X\backslash X_1''/X_2''\backslash X_1'/X_2'\cong B_X \backslash X_1'\cup X_1''/X_2'\cup X_2'' \cong B_X\backslash X_1/X_2$. As $B_X$ has a minor $M(K_4)$ then $S_X\backslash X_1'/X_2' \cong M(K_4)$.\\
	If $X_1' = X_2'=\emptyset $. Then (i) holds. \\
	If $X_1'=\emptyset$ and $X_2' \neq \emptyset$. Then (ii) holds.\\
	If $X_1' \neq \emptyset$. $X_1' \subseteq X$ and $|X_1'| \leq k$. Now, if $|X_1'|= k$, then $X_1'=X$ and $X_2'=\emptyset$ then $S_X\backslash X_1' \cong S_X\backslash X \cong S\backslash X \cong M(K_4)$. Thus, $S$ contains $M(K_4)$ minor, hence (iii) holds. Now, if $0 < |X_1'| < k$ then let $x \in X_1'$, $T=X-{x}$ and $T'=X_1'-{x}$ and $N=S\backslash x$ then $N$ is a minor of $B$, thus $N$ is a graphic matroid. $S_X\backslash x \cong (S\backslash x)_{(X-{x})} \cong N_T$. Thus $S_X\backslash X_1'/X_2'= (S\backslash x)_{(X-{x})}\backslash T'/X_2'=N_T\backslash T'/X_2' \cong M(K_4)$, as As $S_X\backslash X_1'/X_2' \cong M(K_4)$ thus $N \in \mathcal{GG}_{k-1}$. As $N=S\backslash x$, $S =\tilde{N}$. Hence (iv) holds.  
\end{proof}

\begin{lem}\label{GG_rel_min_quo}
	Let $S$ be a graphic matroid as stated in Lemma \ref{GG_main_lem} (i) and (ii). Then exists a graphic matroid $Z$ with $q \in Z$, such that $Z\backslash q \cong M(K_4)$ and $S \cong Z/q$ or $S$ is a binary coextension of $Z/q$ by at most $k$ elements. 	
\end{lem}

\begin{proof}
	Let $S$ be a graphic matroid as stated in Lemma \ref{GG_main_lem} (i) and (ii), then for some subset $X'$ of $X$ either $S_X \cong M(K_4)$ or $S_X/X' \cong M(K_4)$. From the Definition \ref{def_Splitting} and \ref{def_element_sp}, $S_X'\backslash q \cong S_X$ and $S_X'/q \cong S$ where $S_X'$ is element splitting and $S_X$ is splitting of $S$ with respect to $X$. \\ 
	Case (i). If $S_X \cong M(K_4)$, then take $Z=S_X'$ then $E(Z)= E(S)\cup q$. Then $Z/q \cong S_X'/q \cong S$ and $S_X'\backslash q \cong Z\backslash q\cong S_X \cong M(K_4)$. Thus $S \cong Z/q$.\\
	Case (ii). If  $S_X/X' \cong M(K_4)$, then take $Z = S_X'/X'$ thus $ Z \backslash q \cong S_X'/X'\backslash q \cong S_X'\backslash q /X' \cong S_X/X'\cong M(K_4)$ and $Z/q = S_X'/X'/q \cong S_X'/q/X' \cong S/X'$. As $Z/q \cong S/X'$ and $|X'| \leq k$ then $S$ is a binary coextension of $Z/q$ by at most $k$ elements.
\end{proof}
\noindent For a matroid $Z$ with $q \in E(Z)$, if $ Z\backslash q \cong B$  then the matroid $Z/q$ is known as a quotient of $B$ and $B$ is an elementary lift of $Z/q$. Thus, by Lemma \ref{GG_main_lem} and Lemma \ref{GG_rel_min_quo}, to find prohibited minor $S$ of a graphic matroid, we need to find $Z/q$ such that $Z\backslash q \cong M(K_4)$. In the following lemma, all graphic quotients of $M(K_4)$ are found. 
\begin{lem} \label{qk4}
	A graphic quotient of $M(K_4)$ is isomorphic to $M(Q_i)$, Figure \ref{Fig_K4_Q} shows graph $Q_i$, for $i=1,2,3,4$. 
\end{lem}	
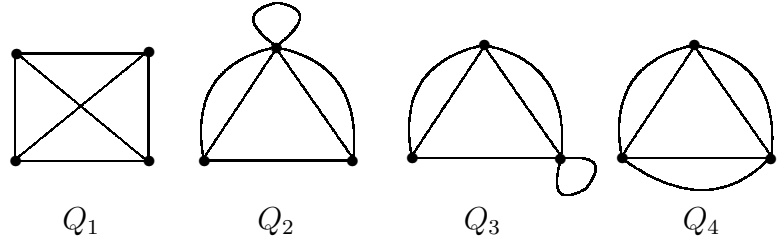
\begin{figure}[h!]
	\centering
\unitlength 1mm 
\linethickness{0.4pt}
\ifx\plotpoint\undefined\newsavebox{\plotpoint}\fi 
\begin{picture}(105.583,31.001)(0,0)
	\put(49.129,10.053){\circle*{1.5}}
	\put(39.129,25.053){\circle*{1.5}}
	\put(29.629,10.084){\circle*{1.5}}
	\put(29.476,10.102){\line(1,0){19.534}}
	\multiput(38.913,25.203)(.0336763754,-.0485889968){309}{\line(0,-1){.0485889968}}
	\multiput(39.062,24.906)(-.0336655052,-.0512787456){287}{\line(0,-1){.0512787456}}
	\qbezier(39.211,25.055)(50.582,22.75)(49.17,10.041)
	\qbezier(29.548,10.189)(27.319,22.379)(38.765,25.055)
	\qbezier(38.913,25.352)(33.116,29.663)(39.211,31.001)
	\qbezier(39.211,31.001)(45.231,30.034)(39.062,25.203)
	\put(76.48,10.345){\circle*{1.5}}
	\put(66.48,25.345){\circle*{1.5}}
	\put(56.98,10.375){\circle*{1.5}}
	\put(56.828,10.394){\line(1,0){19.534}}
	\multiput(66.265,25.495)(.0336731392,-.0485889968){309}{\line(0,-1){.0485889968}}
	\multiput(66.413,25.198)(-.0336655052,-.0512787456){287}{\line(0,-1){.0512787456}}
	\qbezier(66.562,25.346)(77.934,23.042)(76.522,10.332)
	\qbezier(56.9,10.481)(54.67,22.67)(66.116,25.346)
	\qbezier(76.522,10.338)(83.509,11.156)(79.792,6.325)
	\qbezier(79.792,6.325)(74.59,3.426)(76.522,10.338)
	\put(104.129,10.345){\circle*{1.5}}
	\put(94.129,25.345){\circle*{1.5}}
	\put(84.629,10.375){\circle*{1.5}}
	\put(84.477,10.394){\line(1,0){19.534}}
	\multiput(93.914,25.495)(.0336731392,-.0485889968){309}{\line(0,-1){.0485889968}}
	\multiput(94.062,25.198)(-.0336655052,-.0512787456){287}{\line(0,-1){.0512787456}}
	\qbezier(94.211,25.346)(105.583,23.042)(104.171,10.332)
	\qbezier(84.549,10.481)(82.319,22.67)(93.765,25.346)
	\qbezier(84.549,10.338)(96.144,1.865)(104.171,10.635)
	\put(22.38,10.007){\circle*{1.35}}
	\put(4.83,10.035){\circle*{1.35}}
	\put(4.693,10.051){\line(1,0){17.581}}
	\put(4.976,24.182){\circle*{1.35}}
	\put(22.361,24.407){\circle*{1.35}}
	\put(4.873,24.274){\line(1,0){17.392}}
	\put(22.265,24.274){\line(0,-1){14.047}}
	\put(4.74,9.958){\line(0,1){14.717}}
	\multiput(22.326,24.632)(-.0409585253,-.0336981567){434}{\line(-1,0){.0409585253}}
	\multiput(4.835,24.386)(.0405783837,-.0337338371){430}{\line(1,0){.0405783837}}
	\put(13.454,2){\makebox(0,0)[cc]{$Q_1$}}
	\put(39.102,2){\makebox(0,0)[cc]{$Q_2$}}
	\put(66.221,2){\makebox(0,0)[cc]{$Q_3$}}
	\put(95.022,2){\makebox(0,0)[cc]{$Q_4$}}
\end{picture}

\caption{Graphic Quotient of $M(K_4)$}
\label{Fig_K4_Q}
\end{figure}
\begin{proof}
	A quotient of $M(K_4)$ is the matroid $Z/q$ for some binary matroid $Z$ with $q \in E(Z)$ and $Z \backslash q \cong M(K_4)$. As $Z/q$ is graphic then for some connected graph $G$, $Z/q \cong M(G)$. If $\{q\}$ is a cocircuit or a circuit of $Z$ then $Z/q \cong Z \backslash q \cong M(K_4)$. Thus $G \cong Q_1$.

If $\{q\}$ is not a circuit or a cocircuit of $Z$. Then $r(Z \backslash q)=3$ and  $E(Z \backslash q)=6$. Then $r(Z)=3$ and $E(Z)=7$. Thus $E(Z/q)=6$ and $r(Z/q)=2$. Thus there are $6$ arcs and $3$ nodes in the graph $G$.  $G$ is not simple as there is no simple graph with $6$ arcs and $3$ nodes.\\
Suppose, there are more than two multiple arcs in $G$, then $Z \backslash q\cong M(K_4)$ will contain a 2-circuit, as $Z/q$ contains more than two multiple elements, a contradiction. Hence, $G$ can not have more than two multiple arcs. Also, if $G$ contains more than one loop, then $Z\backslash q\cong M(K_4)$ will contain a 2-circuit or a loop, a contradiction. Thus $G$ can not have more than one loop. Hence following are the two cases for $G$. \\
Case (i). Suppose $G$ has one loop, then $G$ is a graph with a loop added to a graph with $5$ arcs and $3$ nodes. By Harary \cite{Har} (page 226), $Q_2$ minus the loop is the only graph on $5$ arcs and $3$ nodes. Hence $G\cong Q_2$ or $G\cong Q_3$.\\
Case (ii). Suppose $G$ does not contain a loop and more than two multiple arcs. Then there is only one graph on $6$ arcs and $3$ nodes given by Harary \cite{Har} (page 226). Hence $G\cong Q_4$.	
\end{proof}
We now prove the main lemma which is used in the paper to prove main theorems.\\
\begin{lem}\label{GG_main_thm}
A graphic matroid $B \in \mathcal{GG}_k$, for $k \geq 2$. Then $B$ consist of a minor say $S$ for which one of the below is satisfied.\\
(i) $S \cong \tilde{N}$, for some minor $N \in \mathcal{GG}_{k-1}$.\\
(ii) $S = M(Q_i)$ or $S$ is a binary coextension of $M(Q_i)$ by no more than $k$ elements, Figure \ref{Fig_K4_Q} shows graph $Q_i$, for $i=1,2,3,4$.
\end{lem}
\begin{proof}
Suppose a graphic matroid $B \in \mathcal{GG}_k$. Thus, by Theorem \ref{gammoid}, for some  $X \subseteq E(B)$ with $|X|\geq 2$, $B_X$ contain a minor $M(K_4)$ . Then by Lemma \ref{GG_main_lem}, $B$ has $S$ a minor that satisfies one of the following.\\
(a). $S_X \cong M(K_4)$;\\
(b). $S_X/X'\cong M(K_4)$ for some $X' \subseteq X$;\\ 
(c). $S$ contains a minor $M(K_4)$;\\
(d). $S\cong \tilde{N}$ for some $N \in \mathcal{GG}_{k-1}$.\\
If minor $S$ satisfies (a) or (b) then by Lemma \ref{GG_rel_min_quo}, $S= Z/q$ or $S$ is a binary coextension of $Z/q$ by no more than $k$ elements whenever $Z \backslash q \cong M(K_4)$, for some matroid $Z$ with $q \in E(Z)$ and by Lemma \ref{qk4}, $Z/q \cong M(Q_i)$, Figure \ref{Fig_K4_Q} shows graph $Q_i$, for $i=1,2,3,4$. Thus, $S=M(Q_i)$ or $S$ is a binary coextension of $M(Q_i)$ not more than $k$ elements, for $ i=1,2,3,4$. Hence (ii) holds.  If $S$ contains a minor $M(K_4)$, as $Q_1=M(K_4)$,  hence (ii) holds.
If minor $S$ satisfies (d), then (i) holds.  
\end{proof}

\section{Splitting of Graphic Matroids That Gives a Binary Gammoid After Splitting}

\noindent  In this section prohibited minors for graphic matroids whose splitting is a binary gammoid are obtained.

\begin{proof}[\bf{Proof of Theorem \ref{main_cor}}]
	If a graphic matroid contains $M(Q_i)$ minor, then proving $B_X$ is not a binary gammoid is easy, for $i=1,2,3,4$ and for some $X \subseteq E(B)$ with $|X| \geq 2$.\\
	Conversely, Suppose $B$ does not contain $M(Q_i)$ minor, for $i=1,2,3,4$. We need to prove that $B_X$ is a binary gammoid. On the contrary, if $B_X$ is not a binary gammoid. Then, $B \in \mathcal{GG}_k$ for some $k$. Thus by Lemma \ref{GG_main_thm}, $B$ contains a $M(Q_i)$ minor for $i=1,2,3,4$, a contradiction. Thus $B_X$ is a binary gammoid. 
\end{proof}

\begin{lem}\label{GG_1_lemma}
	The class $\mathcal{GG}_1$ is empty.
\end{lem}
\begin{proof}
	The proof is straightforward.
\end{proof}

Using Lemma \ref{GG_main_thm}, we now find minimal minors of the class $\mathcal{GG}_2$ and $\mathcal{GG}_3$.

\begin{proof}[\bf{Proof of Theorem \ref{thm_GGm2}}]
	If a graphic matroid $B$ contains a minor $M(G_1)$ or $M(G_2)$ then $B_{x,y} $ is not a binary gammoid as $M(G_i)_{x,y}/x \cong M(K_4)$ for $i=1,2$, Figure \ref{fig_GGm_2} shows the graphs $G_1$, $G_2$ and $x,y$.
	
	Conversely, suppose $B$ do not contain a minor $M(G_1)$ or $M(G_2)$, then we prove that $M \notin \mathcal{GG}_2$. On contrary, suppose, $B \in \mathcal{GG}_2$ , then by Lemma \ref{GG_main_thm}, $B$ has minor $S$ such that $S \cong \tilde{N}$ where $N \in \mathcal{GG}_{k-1}$ or $S=M(Q_i)$ or $S$ is a binary coextension of $M(Q_i)$ by $1$ or $2$ elements, Figure \ref{Fig_K4_Q} shows the graph $Q_i$, for $i=1,2,3,4$. Note that $\mathcal{GG}_1=\phi$. Thus $S \cong M(Q_i)$ or a binary coextension of $M(Q_i)$ by $1$ or $2$ elements. Let $S\cong M(G)$ for some connected graph $G$, as $S$ is a graphic matroid. If $G$ has a 2-edge cut. Then $S$ has a cocircuit say $\{a,b\}$ and for some $i$, it has a cocircuit of $Q_i\backslash a\cong M(K_4)$, a contradiction. Hence $G$ can not have a 2-edge cut. Let $S$ is a binary coextension of $M(Q_i)$ by $y$, where $y$ is a cocircuit then $S_{x,y}\cong M(K_4)$ has a cocircuit $\{x,y\}$ for any $x \in E(S)$, a contradiction. Hence $S$ cannot be a coextension of $M(Q_i)$ by a cocircuit. Also, if $S$ is a coextension by a loop such that $S$ contains more than one loop, then splitting matroid will contain a 2-cocircuit or a loop. Hence we take a coextension that can not have more than one loop.\\
	Case (i) If $S \cong M(Q_1)$ then $S_{x,y} \ncong M(K_4)$ for any $\{x,y\}\in E(S)$, thus $S$ is a binary coextension of $M(Q_1)$ by $1$ element or $2$ elements not containing a 2-edge cut. $G_2$ is the only coextension by a loop. Thus $S\cong M(G_2)$, Hence we discard $M(Q_1)$.\\
	Case (ii) If $S \cong M(Q_2)$ then $S_{x,y}$ will either contain a pair of parallel arcs or a loop, for any $\{x,y\} \subseteq E(S)$. Hence $S_{x,y} \ncong M(K_4)$ for any $\{x,y\} \in E(S)$, thus $S$ is a binary coextension of $M(Q_2)$ by $1$ or $2$ elements not containing 2-edge cut and more than one loop, such coextensions of $Q_2$ are shown in Figure \ref{extension}. Note that, $W_1 \cong G_1$. Hence we discard $W_1$. Also $M(W_i)_{x,y} \ncong M(K_4)$ for any $x,y \in E(M(W_i))$ for $i=2,3$, hence we take coextensions of $W_2$ and $W_3$ by one element not containing a 2-edge cut and more than one loop.  

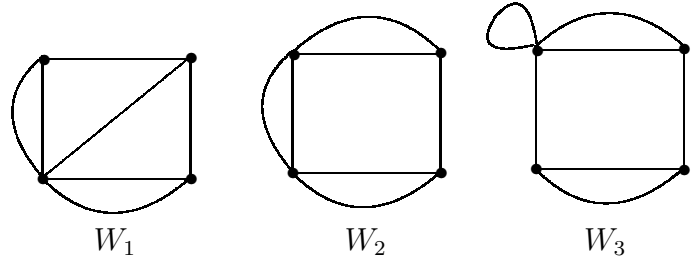
\begin{figure}[h!]
\centering
\unitlength 1mm 
\linethickness{0.4pt}
\ifx\plotpoint\undefined\newsavebox{\plotpoint}\fi 
\begin{picture}(94.435,35.75)(0,0)
	\put(28.81,9.25){\circle*{1.5}}
	\put(9.31,9.281){\circle*{1.5}}
	\put(9.158,9.299){\line(1,0){19.534}}
	\put(9.473,25){\circle*{1.5}}
	\put(28.79,25.25){\circle*{1.5}}
	\put(9.358,25.102){\line(1,0){19.325}}
	\put(28.683,25.102){\line(0,-1){15.608}}
	\put(9.21,9.196){\line(0,1){16.352}}
	\qbezier(9,9.5)(18.125,-0)(28.75,9.5)
	\qbezier(9.25,25.5)(1.375,18.875)(9,9.75)
	\multiput(28.75,25.5)(-.0409751037,-.0337136929){482}{\line(-1,0){.0409751037}}
	\put(61.685,10){\circle*{1.5}}
	\put(42.185,10.031){\circle*{1.5}}
	\put(42.033,10.049){\line(1,0){19.534}}
	\put(42.348,25.75){\circle*{1.5}}
	\put(61.665,26){\circle*{1.5}}
	\put(42.233,25.852){\line(1,0){19.325}}
	\put(61.558,25.852){\line(0,-1){15.608}}
	\put(42.085,9.946){\line(0,1){16.352}}
	\qbezier(41.875,10.25)(51,.75)(61.625,10.25)
	\qbezier(42.125,26.25)(34.25,19.625)(41.875,10.5)
	\put(93.685,10.5){\circle*{1.5}}
	\put(74.185,10.531){\circle*{1.5}}
	\put(74.033,10.549){\line(1,0){19.534}}
	\put(74.348,26.25){\circle*{1.5}}
	\put(93.665,26.5){\circle*{1.5}}
	\put(74.233,26.352){\line(1,0){19.325}}
	\put(93.558,26.352){\line(0,-1){15.608}}
	\put(74.085,10.446){\line(0,1){16.352}}
	\qbezier(73.875,10.75)(83,1.25)(93.625,10.75)
	\qbezier(42,26)(52.25,35.375)(61.5,26.25)
	\qbezier(74,26.75)(83.625,35.75)(93.75,26.75)
	\qbezier(73.75,27)(65.25,25.125)(68.75,30.75)
	\qbezier(68.75,30.75)(73,35.625)(74.25,27)
	\put(18.75,1){\makebox(0,0)[cc]{$W_1$}}
	\put(51.75,1){\makebox(0,0)[cc]{$W_2$}}
	\put(83.25,1){\makebox(0,0)[cc]{$W_3$}}
\end{picture}
\caption{Coextensions of $Q_2$ by one element}
\label{extension}
\end{figure}
\noindent If we take a coextension of $W_2$ by a loop then splitting of the coextension matroid will either contain a loop or a 2-circuit. Hence, we do not consider the coextension of $W_2$ with a loop. Thus, The graphs $V_1$ and $V_2$ are coextensions of $W_2$ and $V_3$ is a coextension of $W_3$, where $V_1, V_2, V_3$ are given in Figure \ref{extension2}. Coextensions $V_1$, $V_2$ and $V_3$ contains minor $G_1$, hence we discard $V_1, V_2, V_3$ and hence $Q_2$.

\begin{figure}[h!]
\centering
\unitlength 1mm 
\linethickness{0.4pt}
\ifx\plotpoint\undefined\newsavebox{\plotpoint}\fi 
\begin{picture}(99.06,36)(0,0)
	\put(33.435,7){\circle*{1.5}}
	\put(13.935,7.031){\circle*{1.5}}
	\put(13.783,7.049){\line(1,0){19.534}}
	\put(14.098,22.75){\circle*{1.5}}
	\put(33.415,23){\circle*{1.5}}
	\put(13.983,22.852){\line(1,0){19.325}}
	\put(33.308,22.852){\line(0,-1){15.608}}
	\put(13.835,6.946){\line(0,1){16.352}}
	\qbezier(13.875,23.25)(6,16.625)(13.625,7.5)
	\put(66.31,7.75){\circle*{1.5}}
	\put(46.81,7.781){\circle*{1.5}}
	\put(46.658,7.799){\line(1,0){19.534}}
	\put(46.973,23.5){\circle*{1.5}}
	\put(66.29,23.75){\circle*{1.5}}
	\put(46.858,23.602){\line(1,0){19.325}}
	\put(66.183,23.602){\line(0,-1){15.608}}
	\put(46.71,7.696){\line(0,1){16.352}}
	\qbezier(46.75,24)(38.875,17.375)(46.5,8.25)
	\put(98.31,8.25){\circle*{1.5}}
	\put(78.81,8.281){\circle*{1.5}}
	\put(78.658,8.299){\line(1,0){19.534}}
	\put(78.973,24){\circle*{1.5}}
	\put(98.29,24.25){\circle*{1.5}}
	\put(78.858,24.102){\line(1,0){19.325}}
	\put(98.183,24.102){\line(0,-1){15.608}}
	\put(78.71,8.196){\line(0,1){16.352}}
	\qbezier(78.5,8.5)(87.625,-1)(98.25,8.5)
	\put(23.25,33.5){\circle*{1.5}}
	\put(56.25,34.5){\circle*{1.5}}
	\put(88,35.25){\circle*{1.5}}
	\multiput(13.5,23)(.0337370242,.0371972318){289}{\line(0,1){.0371972318}}
	\multiput(23.25,33.75)(.0336700337,-.0353535354){297}{\line(0,-1){.0353535354}}
	\multiput(46.75,24)(.0336879433,.0381205674){282}{\line(0,1){.0381205674}}
	\multiput(56.25,34.75)(.0337370242,-.0371972318){289}{\line(0,-1){.0371972318}}
	\multiput(78.75,24.5)(.0336363636,.04){275}{\line(0,1){.04}}
	\multiput(88,35.5)(.0336700337,-.037037037){297}{\line(0,-1){.037037037}}
	\qbezier(13.75,23)(11.625,32.875)(23,33.25)
	\qbezier(33.25,23.5)(40.75,17.125)(33.25,7.25)
	\qbezier(78.75,24.75)(77.625,34.875)(88,35.5)
	\multiput(56,35)(.0336700337,-.0909090909){297}{\line(0,-1){.0909090909}}
	\put(24,1){\makebox(0,0)[cc]{$V_1$}}
	\put(56.75,1){\makebox(0,0)[cc]{$V_2$}}
	\put(87.5,1){\makebox(0,0)[cc]{$V_3$}}
\end{picture}

\caption{Coextensions of $W_1$ and $W_2$}
\label{extension2}
\end{figure}
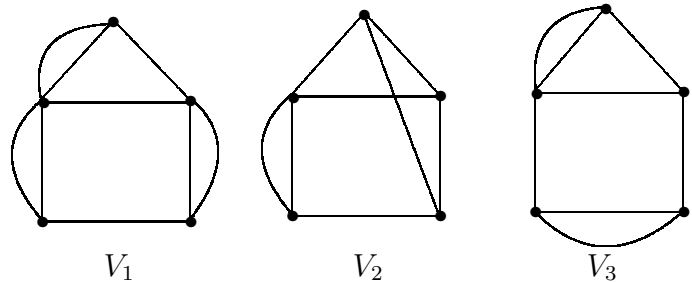   
Case (iii) If $P \cong M(Q_i)$, then on the same line discussed above for $Q_2$, we can discard $Q_i$, for $i=3,4$.  

Thus, from above we conclude that $M \notin \mathcal{GG}_2$.
\end{proof}

\noindent In the following theorem we find minors for $\mathcal{GG}_3$.

\begin{proof}[\bf{Proof of Theorem \ref{thm_GGm3}}]
Suppose a graphic matroid $B$ has a $M(G_i)$ minor, Figure \ref{fig_GGm_3} shows graph $G_i$, for $i=3,4,5,6$. \\
Let $X=\{x,y,z\}$, Figure \ref{fig_GGm_3} shows $\{x,y,z\}$ in the graph $G_6$. Let a matrix $A_6$ denotes the matroid $M(G_6)$, where $A_6$ is given below.\\
$A_6 =\left[ \begin{array}{cccccc}
	x & y & z & & &  \\
	1 & 0 & 1 &1&0&1\\
	0 & 1 & 1 &0&1&1  \end{array} \right].$\\
Then $ (A_6)_X =\left[ \begin{array}{cccccc}
	x & y & z & & &  \\
	1 & 0 & 1 &1&0&1\\
	0 & 1 & 1 &0&1&1 \\
	1 & 1 & 1 &0&0&0  \end{array} \right].$\\
Obtain a matrix $D$ from $(A_6)_X$ by performing operation $R_1 \rightarrow R_1+R_3$.\\
$ D =\left[ \begin{array}{cccccc}
	x & y & z & & &  \\
	0 & 1 & 0 &1&0&1\\
	0 & 1 & 1 &0&1&1 \\
	1 & 1 & 1 &0&0&0  \end{array} \right].$\\
Here, $M(G_6)_X \cong M(D)\cong M(K_4)$, hence $M(G_6)_X$ is not a binary gammoid. Similarly, if $B$ has a minor $M(G_i)$ for $i=3,4,5$ then $B \in \mathcal{GG}_3$ .

Conversely, suppose $B$ does not contain a $M(G_i)$ minor, for $i=3,4,5,6$. Then we prove that $ B \notin \mathcal{GG}_3$. On the contrary suppose that $B \in \mathcal{GG}_3$. Thus by Lemma \ref{GG_main_thm}, $B$ has a minor $S$ such that either (i) $S \cong \tilde{N}$ where $N$ is minor of class $ \mathcal{GG}_2$. (ii) $S=M(Q_i)$ or $S$ is a coextension of $M(Q_i)$ by $1$, $2$ or $3$ elements, for $i=1,2,3,4$. \\
If $S \cong \tilde{N}$ where $N$ is minor of class $ \mathcal{GG}_2$. Then by Theorem \ref{thm_GGm2}, $M(G_1)$ and $M(G_2)$ are two minimal minors of the class $\mathcal{GG}_2$. Note that $M(G_1)$ has a minor $M(G_4)$ and $M(G_2)$ has a minor $M(G_3)$, a contradiction, hence case (i) discarded. Now, if $S=M(Q_i)$ or $S$ is a coextension of $M(Q_i)$ by at most 3 elements, for $i=1,2,3,4$. Then note that $M(Q_1)=M(G_3)$, $M(Q_2)=M(G_4)$, $M(Q_3)=M(G_5)$ and $M(Q_4)=M(G_6)$. Thus we discard case (ii). Hence, $ B \notin \mathcal{GG}_3$. 
\end{proof}

\section{Graphic Matroids Whose Element and es-splitting is a Binary Gammoid}
In this section, we obtain prohibited minors for graphic matroids whose element splitting and es-splitting is a binary gammoid.

\begin{thm}\label{thm_Esp}
Let $B$ be a graphic matroid. Then $B_X'$ is a binary gammoid if and only if $M(G_7)$ is not a minor of $B$, where $G_7$ is as shown in Figure \ref{Fig_es_gr_gam} and $X \subseteq E(B)$ with $|X|\geq 2$. 
\end{thm}  
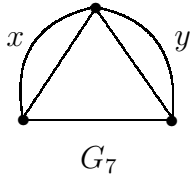
\begin{figure}[h!]
	\centering 
	
	\unitlength 1mm 
	\linethickness{0.4pt}
	\ifx\plotpoint\undefined\newsavebox{\plotpoint}\fi 
	\begin{picture}(28.264,24.188)(0,0)
		\put(26.81,8.438){\circle*{1.5}}
		\put(16.81,23.438){\circle*{1.5}}
		\put(7.31,8.469){\circle*{1.5}}
		\put(7.158,8.487){\line(1,0){19.534}}
		\multiput(16.595,23.588)(.0336731392,-.0485889968){309}{\line(0,-1){.0485889968}}
		\multiput(16.743,23.291)(-.0336655052,-.0512787456){287}{\line(0,-1){.0512787456}}
		\qbezier(16.892,23.44)(28.264,21.135)(26.852,8.426)
		\qbezier(7.23,8.574)(5,20.764)(16.446,23.44)
		\put(17.25,3.25){\makebox(0,0)[cc]{$G_7$}}
		\put(6.25,19.25){\makebox(0,0)[cc]{$x$}}
		\put(28.25,19){\makebox(0,0)[cc]{$y$}}
	\end{picture}
	\caption{Prohibited minor for graphic matroids whose element splitting is a binary gammoid. }
	\label{Fig_es_gr_gam}
\end{figure}
 \begin{proof}
	Suppose, a graphic matroid $B$ has a $M(G_7)$ minor, then proving $B_X'$ is not a binary gammoid is easy, for $X=\{x,y\}$. Conversely, suppose $B$ does not contain $M(G_7)$, then we will prove that $B_X'$ is a gammoid. On the contrary suppose not, then $B_X'$ will contain $M(K_4)$ minor, by Lemma \ref{gammoid}. Thus, $B_X'\backslash X_1/X_2=M(K_4)$, for some subset $X_1$ and $X_2$ of $E(B)\cup \{q\}$. There are three cases as given below.\\
Case (i). If $q\notin X_1 \cup X_2$ then $B_X'\backslash X_1/X_2/q=B_X'/q\backslash X_1/X_2=M(K_4)/q$, thus $B\backslash X_1/X_2=M(G_7)$ since $B_X'/q=B$. Thus $B$ has a $M(G_7)$ minor of $B$, a contradiction.\\
Case (ii). If $q\in X_1$ then $B_X'\backslash q \backslash \{X_1-q\}/X_2=B_X\backslash \{X_1-q\}/X_2=M(K_4)$. Thus $B_X$ is a not binary gammoid, hence by Theorem \ref{main_cor},  $M(Q_i)$ is a minor of $B$ for $i=1,2,3,4$ and each of the minor $M(Q_i)$ contain $M(G_7)$, a contradiction.\\
Case (iii). If $q \in X_2$ then $B_X'/q \backslash X_1/\{X_2-q\}=B \backslash X_1/\{X_2-q\}=M(K_4)$, since $B_X'/q=B$, then $B$ contains a minor $M(K_4)$ which has a minor $M(G_7)$, which is a contradiction. \\
Thus from all the cases discussed above, $B_X'$ is a binary gammoid. Hence the result.   
\end{proof}
We now obtain prohibited minors for graphic matroids whose es-splitting is a gammoid.

\begin{thm}\label{thm_es_sp}
Let $B$ be a graphic matroid. Then $B_X^e$ is a binary gammoid if and only if $B$ does not contain a $M(G_8)$ minor, where $G_8$ is as shown in Figure \ref{Fig_essp}  and $X \subseteq E(B)$ with $|X|\geq 2$ and $e \in H$. 
\end{thm}  
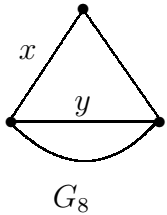
\begin{figure}[h!]
	\centering
	\unitlength 1mm 
	\linethickness{0.4pt}
	\ifx\plotpoint\undefined\newsavebox{\plotpoint}\fi 
	\begin{picture}(27.81,26.688)(0,0)
		\put(27.06,10.938){\circle*{1.5}}
		\put(17.06,25.938){\circle*{1.5}}
		\put(7.56,10.969){\circle*{1.5}}
		\put(7.408,10.987){\line(1,0){19.534}}
		\multiput(16.845,26.088)(.0336731392,-.0485889968){309}{\line(0,-1){.0485889968}}
		\multiput(16.993,25.791)(-.0336655052,-.0512787456){287}{\line(0,-1){.0512787456}}
		\put(15.5,.75){\makebox(0,0)[cc]{$G_8$}}
		\qbezier(26.75,11)(17.5,.5)(7.25,11)
		\put(9.75,20.25){\makebox(0,0)[cc]{$x$}}
		\put(17,13){\makebox(0,0)[cc]{$y$}}
	\end{picture}
	\caption{Prohibited minor for graphic matroids whose es-splitting is a binary gammoid}
	\label{Fig_essp}
\end{figure}
\begin{proof}
	Suppose  $M(G_8)$ is a minor of $B$, then it is very easy to prove that $B_X^e$ is not gammoid for $X=\{x,y\}$ and $e \in X$.\\
Conversely, suppose $X$ does not have minor $M(G_8)$. then we will prove that $B_X^e$ is a binary gammoid. Suppose not, then $B_X^e$ contain minor $M(K_4)$. As $B$ is graphic then for some connected graph $G$, $B=M(G)$. Now let a graph $R$ be obtained from $G$ adding one parallel edge. From the Definition \ref{def_es} it is clear that $M(R)_X'=B_X^e$ then it follows that $M(R)_X'$ is not a gammoid. Thus, $M(G_7)$ is a minor of $B$, by Theorem \ref{thm_Esp}. And, $M(G_8)$ is a minor of $M(G_7)$, which is a contradiction. Thus $B_X^e$ is a binary gammoid and hence the result.
\end{proof}

\section{Cographic Matroids Whose splitting is a Binary Gammoid}
We now characterize cographic matroid whose splitting is a binary gammoid. In section 3 we have characterized graphic matroids whose splitting is binary gammoid. The prohibited minors for cographic matroids whose splitting is a binary gammoid will depend on the quotients of $M(K_4)$. Note that $M(K_4)$ can have graphic and non-graphic quotients. In Section 2 we have obtained all graphic quotients of $M(K_4)$ and in the following lemma, we show that all quotients of $M(K_4)$ are graphic.
\begin{lem}\label{non gr lemma}
All quotients of $M(K_4)$ are graphic. 
\end{lem} 
\begin{proof}
Suppose $Z$ be a binary matroid with $Z\backslash q=M(K_4)$, for $q \in E(Z)$. Let the quotient of $M(K_4)$ be denoted by $Q=Z/q$.
Case (i). If $q$ is a loop or a coloop then $Z\backslash q=M(K_4)$ then $Z\backslash q=Z/q =M(K_4)$. Here $M(K_4)$ is graphic. 
Case (ii). If $q$ is not a loop or a coloop. As  $r(Z\backslash q)=3$ and $E(Z\backslash q)=6$ then $r(Z)=3$ and $E(Z)=7$.  Thus, $r(Q)=2$ and $E(Q)=6$. Suppose $Q$ is not graphic. Then, by Theorem \ref{graphic}, $Q$ contains minor $F \in \{F_7, F_7^*, M^*(K_5), M^*(K_{3,3})\}$. A contradiction, as $r(Q)=2$ and $r(F)\geq 3$. Hence $Q$ is graphic. \\
Thus from both cases, we say that all the quotients of $M(K_4)$ are graphic. 
\end{proof}
\noindent In section 3 we have obtained prohibited minor from the graphic quotient of $M(K_4)$ and from Lemma \ref{non gr lemma}, it is clear that all quotients of $M(K_4)$ are graphic hence we conclude that minors obtained in Theorem \ref{main_cor}, are the only prohibited minors for the class of cographic matroids that yields binary gammoid. 

\bibliographystyle{amsplain}

\end{document}